\documentclass[reqno0,12pt]{amsart}
\pdfoutput=1
\usepackage{bbm}
\usepackage[nodate]{datetime}
\usepackage[hmargin=35mm,top=30mm,bottom=30mm,a4paper]{geometry}
\usepackage{color}
\usepackage{fancyhdr}
\usepackage{amsmath,amssymb,amsfonts,amsthm}
\usepackage{amstext}
\usepackage{amsmath}
\usepackage{amssymb}
\usepackage{enumerate}
\usepackage{amsbsy}
\usepackage{amsopn}
\usepackage{bbm,amsthm}
\usepackage{amscd}
\usepackage{amsxtra}
\usepackage{enumitem}
\usepackage{hyperref}
\usepackage{todonotes}

\newtheorem{theorem}{Theorem}[section]
\newtheorem*{theorem*}{Theorem}
\newtheorem{lemma}{Lemma}[section]

\newtheorem{corollary}{Corollary}[section]

\theoremstyle{remark}

\newtheorem{remark}{Remark}[section]

\newcommand{\EE}{\mathbb{E}}
\newcommand{\PP}{\mathbb{P}}

\begin{document}
\title{On the positivity of some weighted partial sums of a random multiplicative function}
\author{Marco Aymone}
\begin{abstract}
Inspired by the papers by Angelo and Xu, Q.J Math., 74, pp. 767-777, and improvements by Kerr and Klurman, arXiv:2211.05540,  we study the probability that the weighted sums of a Rademacher random multiplicative function,
$\sum_{n\leq x}f(n)n^{-\sigma}$, are positive for all $x\geq x_\sigma\geq 1$ in the regime $\sigma\to1/2^+$. In a previous paper by Heap, Zhao and the author, and by the author, when $0\leq \sigma\leq 1/2$ this probability is zero. Here we give a positive lower bound for this probability depending on $x_\sigma$ that becomes large as $\sigma\to1/2^+$. The main inputs in our proofs are a maximal inequality based in relatively high moments for these partial sums combined with a Bonami--Halász's moment inequality, and also explicit estimates for the partial sums of non-negative multiplicative functions.  
\end{abstract}
\maketitle
\section{Introduction}
\subsection{Main result and background} Let $f$ be a Rademacher random multiplicative function, i.e., at primes $(f(p))_p$ is a sequence of i.i.d. random variables having probability $1/2$ to be equal either to $\pm 1$, and at the positive integers,
$$f(n)=\mu^2(n)\prod_{p|n}f(p),$$
where $\mu$ is the M\"obius function. 

For a short and almost complete survey on this topic we refer to the last paper by the author on this subject, see Section 1.3 of \cite{aymone_sign_changes_II}.

In a recent paper by Angelo and Xu \cite{angelo_xu_positivity}, motivated by a Turán conjecture on the positivity of the weighted partial sums of the Liouville function $\lambda$ and its connections with the Riemann hypothesis, they proved that the probability that the partial sums
$$\sum_{n\leq x}\frac{f^*(n)}{n}$$
are positive for all $x\geq 1$ is at least $1-10^{-45}$, where $f^*$ is a Rademacher random \textit{completely} multiplicative function. They also provided an upper bound depending on $x$ for the probability that the partial sum up to $x$ is negative, and this was later improved by Kerr and Klurman, see \cite{klurman_positivity}.

In a paper by Heap, Zhao and the author \cite{aymone_sign_changes}, and by the author \cite{aymone_sign_changes_II}, it was proved that for each $0\leq \sigma\leq 1/2$, the partial sums
\begin{equation}\label{equation weighted sums}
\sum_{n\leq x}\frac{f(n)}{n^\sigma}
\end{equation}
change its sign infinitely often as $x\to\infty$ almost surely. In particular, the probability that these partial sums are positive for all $x\geq x_0$ is zero, for any $x_0\geq 1$.

The aim of this paper is to study the probability that, for each fixed $\sigma>1/2$, the partial sums in \eqref{equation weighted sums} are positive for all $x\geq x_\sigma$, where $x_{\sigma}\geq 1$ may depend on $\sigma$. 

We study this range because of the following facts.

\begin{enumerate}
\item As $\sigma\to1/2^{+}$, $\sum_{n=1}^\infty\frac{f(n)}{n^\sigma}\to 0$ almost surely, so we may have a considerable possibility of a number of sign changes before the partial sums stabilizes in the limit.
\item As said before, almost surely $\sum_{n\leq x}\frac{f(n)}{n^{1/2}}$ changes its sign infinitely often as $x\to\infty$.
\end{enumerate}

Our main result states the following.
\begin{theorem}\label{theorem principal} Let $f$ be a Rademacher random multiplicative function, $\delta>0$ and $\theta$ a fixed parameter in the interval $(0,1)$. Then, there exists $\sigma_0\in(1/2,1]$ such that for all $1/2<\sigma\leq \sigma_0$ and
$x\geq 1$ defined by 
$$\log x =\frac{1}{(\sigma-1/2)^{1/\theta}},$$
the probability that the partial sums $\sum_{n\leq y}\frac{f(n)}{n^\sigma}$ are positive for all $y>x$
is at least
$$1-\exp\left(-\frac{1+o(1)}{2\theta}\frac{(\log x)^{2-2\theta}}{(\log\log x)^{1+2\delta}}\right).$$
\end{theorem}

\begin{remark}
It is interesting to observe that, under complete \textit{independence} and \textit{symmetry}, that is, for i.i.d and symmetric random variables $(X_n)_n$ (Rademacher random variables for example), we have that for all $x\geq 1$, 
$$\sum_{n\leq x}\frac{X_n}{n^\sigma}$$
has equal probability of being positive or negative.
\end{remark}

We obtain the immediate consequence from the theorem above.

\begin{corollary} The probability that the partial sum $\sum_{n\leq y}\frac{f(n)}{n^\sigma}$ is negative for some $y>x$, where $x$ is as in Theorem \ref{theorem principal}, is at most
$$\exp\left(-\frac{1+o(1)}{2\theta}\frac{(\log x)^{2-2\theta}}{(\log\log x)^{1+2\delta}}\right).$$
\end{corollary}

\begin{remark}\label{remark} It is interesting to compare our bound in the corollary above with that of Angelo and Xu \cite{angelo_xu_positivity} or with that of Kerr and Klurman \cite{klurman_positivity}. Indeed, for a Rademacher random completely multiplicative function $f^{*}$, and for $\sigma=1$, the bound that they obtain for the probability that $\sum_{n\leq x}\frac{f^*(n)}{n}$ is negative is exponentially small compared with our bound, that is, it is bounded by
$$\exp\left(-\exp\left(\beta \frac{\log x}{\log\log x}\right)\right).$$

Would our results and that of Angelo and Xu or that of Kerr and Klurman indicate a transition on this probability when one pass from $\sigma=1$ to $\sigma<1$?
\end{remark}

\subsection{Recent progress on sign changes of random multiplicative functions} We finish this introduction by mentioning some interesting progress in this field. 

In the paper \cite{geis_counting_sign_changes} by Geis and Hiary, building upon \cite{aymone_sign_changes}, the authors were able to provide a lower bound for the counting of the number of sign changes of $\sum_{n\leq x}f(n)$ as a function of $x$, almost surely. 

A few months later, Klurman, Lamzouri and Munsch \cite{klurman_counting_sign_changes} studied sign changes of short character sums and real zeros of Fekete polynomials. In this paper they provided a general approach that may be used to prove quantitative bounds in the counting of sign changes of partial sums of arithmetic functions in Analytic Number Theory. Their method apply to Rademacher random multiplicative functions and, as stated in their Remark 1.5, this may be used to improve the result by Geis and Hiary.
 
By considering averages, in the paper \cite{aymone_sign_changes_III}, the author was able to give a lower bound on the expected number of sign changes of $\sum_{n\leq x}f(n)$.

\section{Preliminaries}
\subsection{Notation} 
In this subsection we make a summary of all recurrent notation used in this paper. We hope that it may be useful for the reader always when he or she becomes overloaded with the plenty number of notation used here.

\subsubsection{Letters appearing throughout the text} We will let the letter $p$ to always represent a generic prime number, $n$ to represent a positive integer, $x$ and $y$ real variables used as the edge of an index of summation. $m$ is reserved for the size of the moments of certain random variables. The real variable $\sigma>1/2$ denotes exclusively the point where we are studying our truncated random Dirichlet series. We let $c_j$ to represent positive auxiliary constants and greek letters for the ultimate constants. $\lambda$ will be used as a threshold in a event where a random variable is above or beyond this threshold. The letter $f$ represents our Rademacher random multiplicative function. $\PP$ is the probability of an event.

\subsubsection{Functions in Number Theory} Here $\mu$ denotes the M\"obius function which is defined to be the multiplicative function such that at primes $p$, $\mu(p)=-1$, and at prime powers $p^k$, $k\geq 2$, $\mu(p^k)=0$. The Liouville $\lambda$ is the completely multiplicative function such that at primes $\lambda(p)=-1$. $\omega(n)$ denotes the number of distinct primes that divide $n$.

\subsubsection{Asymptotic notation} We use the standard Vinogradov notation $f(x)\ll g(x)$ or Landau's $f(x)=O(g(x))$ whenever there exists a constant $c>0$ such that $|f(x)|\leq c|g(x)|$, for all $x$ in a set of parameters. When not specified, this set of parameters is an infinite interval $(a,\infty)$ for sufficiently large $a>0$. Sometimes is convenient to indicate the dependence of this constant in other parameters. For this, we use both $\ll_\delta$ or $O_\delta$ to indicate that $c$ may depends on $\delta$. The standard $f(x)=o(g(x))$ means that $f(x)/g(x)\to0$ when $x\to a$, where $a$ could be a complex number or $\pm \infty$. 
\subsection{A combination of a maximal and a moment inequality}
\subsection{Maximal Inequalities} In the study of partial sums of a general set of random variables $(X_n)_n$ frequently one needs to have upper bounds for the probability
$$\PP\left(\max_{N\leq x\leq M}\left|\sum_{N\leq n\leq x}X_n\right|\geq \lambda\right).$$
In chapter I of the book of Pe\~na and Gin\'e \cite{pena_decoupling}, there are a plenty number of these inequalities when we have \textit{independence}. In the absence of independence the situation can be intricate and in certain cases such inequalities cannot be much better than the probability given by the union bound.

The main maximal inequality used in this paper is the following Theorem 10.2 of the book of Billingsley \cite{billingsley_convergence_prob_meas}: Let $S_N=X_1+...+X_N$. Suppose that $\alpha>1/2$ and that $\beta\geq 0$. Let $u_1,...,u_N$ be non-negative real numbers such that for all $\lambda>0$
\begin{equation*}
\PP(|S_j-S_i|\geq \lambda)\leq\frac{1}{\lambda^{4\beta}}\left(\sum_{i<n\leq j}u_n\right)^{2\alpha},\quad 0\leq i\leq j \leq N.
\end{equation*}
Then
\begin{equation*}
\PP\left(\max_{k\leq N}|S_k|\geq \lambda \right)\leq K'_{a,\beta}\frac{1}{\lambda^{4\beta}}\left(\sum_{n\leq N}u_n\right)^{2\alpha},
\end{equation*}
where $K'_{\alpha, \beta}$ is a positive constant that depends only on $\alpha$ and $\beta$. 

\subsubsection{The constant $K'_{\alpha,\beta}$}\label{subsection constant} More precisely, we will need to control the size of the constant $K'_{\alpha, \beta}$ as $\alpha$ and $\beta$ becomes large. 

By a careful inspection\footnote{Our target is the constant of Theorem 10.2 of \cite{billingsley_convergence_prob_meas} which we call here by $K'_{\alpha, \beta}$. This result is obtained as a consequence from Theorem 10.1 of \cite{billingsley_convergence_prob_meas} where it appears the constant $K_{\alpha,\beta}$. Furthermore, Theorem 10.1 is obtained as a consequence of Theorem 10.3 of \cite{billingsley_convergence_prob_meas} whose proof is divided in 4 steps. In step 1 the constant $K_{\alpha,\beta}$ is defined; in step 3, to make the argument work, $K_{\alpha,\beta}$ is replaced by $2^{2\alpha}K_{\alpha,\beta}$.} in the proof of the maximal inequality above, we see that 
$$K'_{\alpha, \beta}=(2^{2\alpha}K_{\alpha,\beta}+1)2^{4\beta},$$
where
$$K_{\alpha, \beta}= \frac{2^{2\alpha}}{C^{4\beta}}\sum_{n=1}^\infty\left(\frac{1}{\theta^{4\beta}2^{2\alpha-1}}\right)^n.$$
Here $\theta\in(0,1)$ is chosen so that the sum above converges and $C$ is chosen so that
$$C\frac{\theta}{1-\theta}=\frac{1}{2}.$$

Thus,
$$K_{\alpha, \beta}= 2^{2\alpha+4\beta}\frac{\theta^{4\beta}}{(1-\theta)^{4\beta}}\sum_{n=1}^\infty\left(\frac{1}{\theta^{4\beta}2^{2\alpha-1}}\right)^n.$$
Now, if $\beta=\alpha\geq 1$, which will be our case below, and $\theta=2^{-1/6}$, then 
$$\theta^{4\beta}2^{2\beta-1}\geq 2^{1/3}>1,$$ 
and hence
$$K_{\beta, \beta}\leq \kappa^{\beta},$$
for all $\beta\geq 1$, for some absolute $\kappa>0$.

Therefore, by abusing a little bit of notation, replacing $\kappa$ by a larger positive constant that we still denote by $\kappa$, we have that
$$K'_{\beta, \beta}\leq \kappa^\beta,$$
for all $\beta\geq 1$.

\subsection{Bonami--Halász's inequality} In \cite{halasz} it was proved by Halász the following inequality: Let $f$ be a Rademacher random multiplicative function, and $(a(n))_{n\leq N}$ be complex numbers. Then
$$\EE \left|\sum_{n\leq N}a(n)f(n)\right|^{m}\leq \left(\sum_{n\leq N}\mu^2(n)|a(n)|^2(m-1)^{\omega(n)}\right)^{m/2},$$
where $m\geq 2$ is any real number and $\omega(n)$ is the number of distinct primes that divide $n$. Actually, in his paper he proves such inequality in a more general context, and comments that he discovered long after the publication process that such inequality was proved before by Bonami \cite{bonami_inequality}. 

\subsection{The combination} We may use the two inequalities stated above to prove the following lemma.

\begin{lemma}\label{lemma maximal inequality}
Let $f$ be a Rademacher random multiplicative function. Let $(a(n))_n$ be a sequence of non-negative real numbers such that 
$$\sum_{n=1}^\infty a^2(n)(\log n)^2<\infty.$$
Let $x\geq 1$, $m\geq 4$ and $\lambda>0$ be any fixed parameters. Then, for some absolute constant $\kappa>0$
$$\PP\left(\sup_{y>x}\left|\sum_{n>y} a(n)f(n)\right|\geq \lambda\right)\leq \frac{\kappa^{m}}{\lambda^m}\left(\sum_{n>x}\mu^2(n)a^2(n)(m-1)^{\omega(n)}\right)^{m/2}.$$
\end{lemma}

\begin{proof} Firstly we observe that, by the Rademacher--Menshov Theorem, see \cite{meaney} for an exposition of Salem's proof of this result, the condition 
$$\sum_{n=1}^\infty a^2(n)(\log n)^2<\infty$$ 
guarantees that the sum
$$\sum_{n=1}^\infty a(n)f(n)$$
converges almost surely.
 
  Set $S_N(x)=\sum_{n>N}a(n)f(n)$. By Markov's bound, for any $1\leq j <i$
$$\PP(|S_j(x)-S_i(x)|\geq \lambda)\leq \frac{1}{\lambda^m}\EE|S_j(x)-S_i(x)|^m.$$
Now, by the Bonami--Halász's inequality
$$\EE|S_j(x)-S_i(x)|^m\leq \left(\sum_{j<n\leq i}\mu^2(n)a^2(n)(m-1)^{\omega(n)}\right)^{m/2}.$$
Thus the hypothesis of the aforementioned inequality in the book of Billingsley \cite{billingsley_convergence_prob_meas} are satisfied with $2\alpha=m/2$, $4\beta=m$ and $u_n=\mu^2(n)a(n)^2(m-1)^{\omega(n)}$. Further, the existence of the absolute constant $\kappa>0$ is guaranteed by the discussion in Section \ref{subsection constant}. \end{proof}

\subsection{Hoeffding's inequality} This is a classical inequality used to estimate large tail probabilities of the sum of independent random variables. In our case, we need a version for the sum of an infinite quantity of random variables. More precisely, we will need the following bound:
\begin{lemma}[Hoeffding's inequality]\label{lemma hoeffding} Let $(f(p))_p$ be Rademacher i.i.d. random variables. Then, for any $\lambda>0$ and any $\sigma>1/2$, when $\sigma\to1/2^+$ we have
$$\PP\left(\sum_{p}\frac{f(p)}{p^{\sigma}}\geq \lambda\right)\leq\exp\left(-\frac{\lambda^2}{2}\frac{1+o(1)}{\log((\sigma-1/2)^{-1})}\right).$$
\end{lemma}

For a proof we refer to \cite{aymoneLIL}, Lemma 2.3. Actually, it is implicit in our lemma above the fact that, as $\sigma\to 1/2^+$,
\begin{equation}\label{equation estimate sum over prime powers}
\sum_p\frac{1}{p^{2\sigma}}=(1+o(1))\log((\sigma-1/2)^{-1}).
\end{equation} 

\section{Some explicit Number Theory}
In this section we devote our attention to estimate the right-handside of the upper bound for the probability given by Lemma \ref{lemma maximal inequality}. For this we will need to make explicit the dependence of this quantity in terms of the parameters involved. So, we ended up by reworking existing arguments in the literature to derive such explicit estimates. As an outcome, we will loose precision but will gain robustness.

A classical result for partial sums of non-negative multiplicative functions is the following: If for some $m>0$
$$\sum_{p\leq x}f(p)\log p\leq mx,$$
for all $x\geq 2$, and
$$\sum_{p^k,\,k\geq 2}\frac{f(p^k)k\log p}{p^k}\leq m,$$
then, there is a constant $c_1>0$ such that for all $x\geq 2$,
$$\sum_{n\leq x}f(n)\leq c_1(m+1)\frac{x}{\log x}\sum_{n\leq x}\frac{f(n)}{n}.$$

A proof of this result can be found in the book of Montgomery--Vaughan \cite{montgomery_livro} (Theorem 2.14)  or in the book of Ramaré \cite{ramare_book} (Theorem 13.1 and see also his references right before the statement of this theorem).

This result allow us to prove the following lemma. 
\begin{lemma}\label{lemma explicite estimate partial sum} There are positive constants $c_3>0$ and $c_5>0$ such that for all $m\geq 4$ and $x\geq 2$,
\begin{equation}\label{equation upper bound 3 omega n}
\sum_{n\leq x}\mu^2(n)(m-1)^{\omega(n)}\leq c_3mx(\log x)^{c_5m}.
\end{equation}
\end{lemma}

Indeed, we know that the asymptotics of the partial sums above can be computed with precision, see for instance the book of Tenenbaum \cite{tenenbaumlivro}, pg. 59 exercise 55 (for a route). But we stress that here we will need inequalities valid in a specific range of $x$, so we need to be careful.
\begin{proof}
Observe that $n\mapsto \mu^2(n)(m-1)^{\omega(n)}$ is a non-negative multiplicative function that vanishes at prime powers bigger than $2$. On the other hand, in the paper \cite{kadiri_explicit} by Fiori, Kadiri and Swidinsky, it is given explicit results for the Chebyshev's prime counting function. We can infer that, for a relatively large constant $c_2$, for all $x\geq 2$

$$\sum_{n\leq x}(m-1)\log p \leq c_2(m-1)x. $$
Therefore, there exists a constant $c_3>0$ such that for all $x\geq 2$
$$\sum_{n\leq x}\mu^2(n)(m-1)^{\omega(n)}\leq c_3m\frac{x}{\log x}\sum_{n\leq x}\frac{\mu^2(n)(m-1)^{\omega(n)}}{n}.$$
The last sum above is bounded by
$$\prod_{p\leq x}\left(1+\frac{m-1}{p}\right).$$
The inequality $1+u\leq e^{u}$ tell us that 
$$\sum_{n\leq x}\mu^2(n)(m-1)^{\omega(n)}\leq c_3m\frac{x}{\log x}\exp\left((m-1)\sum_{p\leq x}\frac{1}{p}\right).$$
By Mertens' estimate, there is a constant $c_4>0$ such that for all $x\geq 10$
$$\sum_{p\leq x}\frac{1}{p}\leq c_4\log\log x.$$
Therefore, there exists a constant $c_5>0$ such that that the claim of the lemma follows. \end{proof}

\begin{lemma}\label{lemma evaluacao da serie} Let $\sigma\in(1/2,1)$. Let $c_5$ be as in Lemma \ref{lemma explicite estimate partial sum}. There are absolute positive constants $c_7,c_8$ such that for all $m\geq 4$ and all $x\geq2$, 
$$\sum_{n>x}\frac{\mu^2(n)(m-1)^{\omega(n)}}{n^{2\sigma}}\leq \frac{c_7^m m^{c_5m}}{(\sigma-1/2)^{c_8m}}\frac{(\log x)^{c_5m}}{x^{2\sigma-1}}.$$
\end{lemma}
\begin{proof} By applying the Riemann--Stieltjes method, we can write the sum as an integral:
\begin{align*}
\sum_{n>x}\frac{\mu^2(n)(m-1)^{\omega(n)}}{n^{2\sigma}}=\int_x^\infty\frac{1}{t^{2\sigma}}d\left(\sum_{n\leq t}\mu^2(n)(m-1)^{\omega(n)}\right).
\end{align*}
By integration by parts we have that the last integral (by applying carefully the limits of integration):
\begin{multline*}
\sum_{n> x}\frac{\mu^2(n)(m-1)^{\omega(n)}}{n^{2\sigma}}=-\frac{1}{x^{2\sigma}}\sum_{n\leq x}\mu^2(n)(m-1)^{\omega(n)}\\+ 2\sigma\int_x^\infty\frac{1}{t^{2\sigma+1}}\left(\sum_{n\leq t}\mu^2(n)(m-1)^{\omega(n)}\right)dt.
\end{multline*}

Using Lemma \ref{lemma explicite estimate partial sum},
$$\sum_{n> x}\frac{\mu^2(n)(m-1)^{\omega(n)}}{n^{2\sigma}}\leq \frac{c_3m(\log x)^{c_5m}}{x^{2\sigma-1}}+ 2c_3m\int_{x}^\infty \frac{(\log t)^{c_5m}}{t^{2\sigma}}dt.$$

Now
\begin{align*}
\int_{x}^\infty \frac{(\log t)^{c_5m}}{t^{2\sigma}}dt&=\sum_{k=0}^\infty \int_{e^k x}^{e^{k+1}x}\frac{(\log t)^{c_5m}}{t^{2\sigma}}dt\\
&\leq \sum_{k=0}^{\infty} \frac{(k+1+\log x )^{c_5m}}{(2\sigma-1)x^{2\sigma-1}e^{k(2\sigma-1)}}\\
&\leq \frac{1}{(2\sigma-1)}\frac{(\log x)^{c_5m}}{x^{2\sigma-1}}\sum_{k=0}^\infty\frac{(k+1)^{c_5m}}{e^{k(2\sigma-1)}}\\
&\leq \frac{1}{(2\sigma-1)}\frac{(\log x)^{c_5m}}{x^{2\sigma-1}}\left(\sup_{k\geq 1}\frac{(k+1)^{c_5m}}{e^{k(\sigma-1/2)}}\right)\sum_{k=0}^\infty\frac{1}{e^{k(\sigma-1/2)}}.
\end{align*}
Using standard calculus, we see that the sup in the last expression above is bounded by
$$c_6^m \frac{m^{c_5m}}{(\sigma-1/2)^{c_5m}}.$$
Using the fact 
$$\frac{1}{1-e^{-(\sigma-1/2})}=\sum_{k=0}^\infty\frac{1}{e^{k(\sigma-1/2)}}$$
and that $e^u=1+u+O(u^2)$ for small $u$, we obtain constants $c_7,c_8>0$ such that
$$\sum_{n>x}\frac{\mu^2(n)(m-1)^{\omega(n)}}{n^{2\sigma}}\leq \frac{c_7^m m^{c_5m}}{(\sigma-1/2)^{c_8m}}\frac{(\log x)^{c_5m}}{x^{2\sigma-1}},$$
for all $x\geq 2$.
\end{proof}

\section{Proof of the main result}
At this point we have established all the necessary tools for the proof of our main result. Before going to the proof, we will establish two more lemmas. The next one, we make implicit the choice of $a(n)=n^{-\sigma}$ in Lemma \ref{lemma maximal inequality}.
\begin{lemma}\label{lemma estimate for the sup} Let $\lambda>0$. For each $0<\theta<1$, there exists $\sigma_0>1/2$ and $\epsilon>0$ such that uniformly for all $\sigma\in(1/2,\sigma_0]$, and
$$\log x = \frac{1}{(\sigma-1/2)^{1/\theta}},$$ 
we have that for some absolute constant $\beta>0$,
$$\PP\left(\sup_{y\geq x}\left|\sum_{n>y}\frac{f(n)}{n^\sigma}\right|\geq \lambda\right)\leq \exp\left(-\beta \frac{(\log x)^{2-2\theta}}{\log\log x}+(\log\lambda ^{-1})\epsilon\frac{(\log x)^{1-\theta}}{\log\log x}\right).$$
\end{lemma}
\begin{proof}
By the previous Lemmas \ref{lemma maximal inequality} and \ref{lemma evaluacao da serie}, for any $\lambda>0$, $m\geq 4$ and all $x\geq 2$ we obtain that
$$\PP\left(\sup_{y\geq x}\left|\sum_{n>y}\frac{f(n)}{n^\sigma}\right|\geq \lambda\right)\leq \frac{\kappa^{m}}{\lambda^m}\left(\frac{c_7^m m^{c_5m}}{(\sigma-1/2)^{c_8m}}\frac{(\log x)^{c_5m}}{x^{2\sigma-1}}\right)^{m/2}.$$

Thus, we are led to optimize
\begin{multline*}
\Lambda:=\exp\bigg{(}c_9m^2\log m+c_9m^{2}\log\log x\\
+c_{10}m^2\log\left(\frac{1}{\sigma-1/2}\right)-mc_{11}(\sigma-1/2)\log x\bigg),
\end{multline*}
where $c_9,c_{10}$ and $c_{11}$ are positive constants such that our target probability is bounded above by $\lambda^{-m}\Lambda$.

In this optimization problem, we begin by choosing $x=x_\sigma$ so that 
\begin{equation}\label{equation auxiliar definition of x} 
\log x = \frac{1}{(\sigma-1/2)^{1/\theta}}, 
\end{equation} 
where $\theta\in(0,1)$ is a fixed parameter. We infer that our optimal solution occurs at
$$m= \epsilon(\sigma-1/2)\frac{\log x}{\log\log x }=\epsilon\frac{(\log x)^{1-\theta}}{\log\log x},$$
for some small fixed $\epsilon>0$ to be chosen shortly.
In particular, we have:
\begin{align*}
c_9m^2\log m &\leq c_9\epsilon^2(1-\theta)\frac{(\log x)^{2-2\theta}}{\log\log x}, \\
c_9m^2\log\log x&=c_9\epsilon^2\frac{(\log x)^{2-2\theta}}{\log\log x},\\
c_{10}m^2\log\left(\frac{1}{\sigma-1/2}\right)&=c_{10}\epsilon^2\theta\frac{(\log x)^{2-2\theta}}{\log\log x},\\
mc_{11}(\sigma-1/2)\log x&=c_{11}\epsilon\frac{(\log x)^{2-2\theta}}{\log\log x}.
\end{align*}

Collecting all these estimates, we obtain that
$$\Lambda\leq \exp\left(w(\epsilon)\frac{(\log x)^{2-2\theta}}{\log\log x}\right),$$
where 
$$w(\epsilon)=\epsilon^2( c_9(1-\theta)+c_9+c_{10}\theta)-\epsilon c_{11}:=\epsilon^2c_{12}-\epsilon c_{11}.$$

We see that $\lim_{\epsilon\to 0^+} w(\epsilon)\epsilon^{-1}=-c_{11}$, so there exists $\epsilon_0>0$ such that
$w(\epsilon_0)<0$. Thus, by setting $-\beta=w(\epsilon_0)$, our target probability is upper bounded
by
$$\frac{1}{\lambda^m}\exp\left(-\beta \frac{(\log x)^{2-2\theta}}{\log\log x}\right)=\exp\left(-\beta \frac{(\log x)^{2-2\theta}}{\log\log x}+(\log\lambda ^{-1})\epsilon\frac{(\log x)^{1-\theta}}{\log\log x} \right).$$
\end{proof}

\begin{lemma}\label{lemma euler product probability} Let $f$ be a Rademacher random multiplicative function, $\delta>0$, $\sigma>1/2$ and $x=x(\sigma,\theta)\geq 1$ as in the previous lemma. Then, for any fixed $\sigma$ sufficiently close to $1/2^{+}$,
$$\PP\left(\sum_{n=1}^\infty\frac{f(n)}{n^\sigma}\geq \exp\left(-2\frac{(\log x)^{1-\theta}}{(\log\log x)^{\delta}}\right)\right)\geq 1-\exp\left(-\frac{1+o(1)}{2\theta}\frac{(\log x)^{2-2\theta}}{(\log\log x)^{1+2\delta}}\right).$$
\end{lemma}
\begin{proof}
By the Euler product representation for the Dirichlet series, and classical estimates for the Riemann $\zeta$ function (see Lemma 2.4 of \cite{aymone_sign_changes}), we have that for all $\sigma>1/2$, almost surely
\begin{equation}\label{equation euler product}
\sum_{n=1}^\infty\frac{f(n)}{n^{\sigma}}=\exp\left(\sum_p\frac{f(p)}{p^{\sigma}}-\frac{1}{2}\log\left(\frac{1}{\sigma-1/2}\right)+O(1)\right).
\end{equation}
Now, by Lemma \ref{lemma hoeffding}, we have that as $\sigma\to1/2^+$,
\begin{align*}
\PP\left(\left|\sum_p\frac{f(p)}{p^{\sigma}}\right|> \frac{(\log x)^{1-\theta}}{(\log\log x)^{\delta}}\right)&\leq 2\exp\left(-\frac{1+o(1)}{2}\frac{((\log x)^{1-\theta})^{2}}{(\log\log x)^{2\delta}}\frac{1}{\log(\sigma-1/2)^{-1}}\right)\\
&\leq \exp\left(-\frac{1+o(1)}{2\theta}\frac{(\log x)^{2-2\theta}}{(\log\log x)^{1+2\delta}}\right).
\end{align*}

Therefore, the event on which
$$\left|\sum_p\frac{f(p)}{p^{\sigma}}\right|\leq \frac{(\log x)^{1-\theta}}{(\log\log x)^{\delta}}$$
has probability at least
$$1-\exp\left(-\frac{1+o(1)}{2\theta}\frac{(\log x)^{2-2\theta}}{(\log\log x)^{1+2\delta}}\right).$$
To conclude, we see that in this event, by \eqref{equation euler product}, 
$$\sum_{n=1}^\infty\frac{f(n)}{n^\sigma}\geq \exp\left(-2\frac{(\log x)^{1-\theta}}{(\log\log x)^{\delta}}\right).$$
\end{proof}

\begin{proof}[Proof of the main result] Let $\mathcal{A}$ be the event in which
$$\sum_{n=1}^\infty\frac{f(n)}{n^\sigma}\geq \exp\left(-2\frac{(\log x)^{1-\theta}}{(\log\log x)^{\delta}}\right)$$
where $x$ is defined as in Lemma \ref{lemma estimate for the sup} and \ref{lemma euler product probability}. Let $\mathcal{B}_x$ be the event in which
\begin{equation}\label{equation auxiliar lambda}
\sup_{y>x}\left|\sum_{n>y}\frac{f(n)}{n^\sigma}\right|\leq \frac{1}{2}\exp\left(-2\frac{(\log x)^{1-\theta}}{(\log\log x)^{\delta}}\right):=\lambda.
\end{equation}

We see that in the event $\mathcal{A}\cap\mathcal{B}_x$ the partial sums
$$\sum_{n\leq y}\frac{f(n)}{n^\sigma}$$
are positive for all $y>x$.
 
Now, by Lemma \ref{lemma euler product probability}, $\mathcal{A}$ has probability at least
$$1-\exp\left(-\frac{1+o(1)}{2\theta}\frac{(\log x)^{2-2\theta}}{(\log\log x)^{1+2\delta}}\right).$$

Let $\lambda$ be as in \eqref{equation auxiliar lambda}. By Lemma \ref{lemma estimate for the sup}, we obtain that the probability of $\mathcal{B}_x$ is at least
\begin{align*}
&1-\exp\left(-\beta \frac{(\log x)^{2-2\theta}}{\log\log x}+\epsilon(\log \lambda^{-1})\frac{(\log x)^{1-\theta}}{\log\log x} \right)\\
=&1-\exp\left(-\beta \frac{(\log x)^{2-2\theta}}{\log\log x}+(2\epsilon+o(1))\frac{(\log x)^{2-2\theta}}{(\log\log x)^{1+\delta}} \right)\\
=&1-\exp\left(-(\beta+o(1)) \frac{(\log x)^{2-2\theta}}{\log\log x}\right).
\end{align*}

In particular, as $\sigma\to1/2^+$
$$\PP(\mathcal{B}_x^c)=o(\PP(\mathcal{A}^c)),$$
where $\cdot^c$ indicates the complementary operation.
Therefore, as $\sigma\to 1/2^+$, 
$$\PP(\mathcal{A}\cap\mathcal{B}_x)\geq 1-\exp\left(-\frac{1+o(1)}{2\theta}\frac{(\log x)^{2-2\theta}}{(\log\log x)^{1+2\delta}}\right),$$
and this completes the proof.
\end{proof}

\section{A concluding remark}

We observe that we can allow $\frac{1}{\theta}\to\infty$, or more generally, to make $\log x$ much larger than any power of 
$(\sigma-1/2)^{-1}$. All we have to do is to rework our proof, and observing that at the end we may have
$$\PP(\mathcal{A})=o(\PP(\mathcal{B}_x)).$$

Our intention here was to make $x$ minimal as possible as a function of $(\sigma-1/2)^{-1}$. We left to the interested reader to outline the details when $\log x$ is is much larger then $(\sigma-1/2)^{-1}$.

\vspace{1cm}

\noindent\textbf{Acknowledgements.} The author is thankful to the anonymous referee for his/her corrections, comments and suggestions that really improved the exposition of this paper. The research of the author is funded by FAPEMIG grant Universal APQ-00256-23 and by CNPq grant Universal 403037 / 2021-2.   

\bibliographystyle{siam}
\bibliography{ay.bib}

\vspace{2cm}

{\small{\sc \noindent
Departamento de Matem\'atica, Universidade Federal de Minas Gerais (UFMG), Av. Ant\^onio Carlos, 6627, CEP 31270-901, Belo Horizonte, MG, Brazil.} \\
\textit{Email address:} \verb|aymone.marco@gmail.com| }

\end{document}